\documentclass[10pt,a4paper,twoside,english]{article}
\usepackage{mathpazo}
\usepackage[T1]{fontenc}
\usepackage[latin9]{inputenc}
\usepackage{color}
\usepackage{amsthm}
\usepackage{amsmath}
\usepackage{amssymb}

\makeatletter

\numberwithin{equation}{section}
\numberwithin{figure}{section}
  \theoremstyle{plain}
  \newtheorem*{assumption*}{Assumption}
  \theoremstyle{plain}
  \newtheorem*{thm*}{Theorem}
 \theoremstyle{definition}
 \newtheorem*{defn*}{Definition}
  \theoremstyle{plain}
  \newtheorem*{prop*}{Proposition}
  \theoremstyle{plain}
  \newtheorem*{lem*}{Lemma}
  \theoremstyle{remark}
  \newtheorem*{rem*}{Remark}

\makeatother

\usepackage{babel}
\addto\extrasfrench{\providecommand{\fg}{\ifdim\lastskip>\z@\unskip\fi~\frqq}}

\begin{document}

\title{Eigenvalues for a Schrödinger operator on a closed Riemannian manifold
with holes}

\author{Olivier Lablée}

\date{12 September 2013}
\maketitle
\begin{abstract}
In this article we consider a closed Riemannian manifold $(M,g)$
and $A$ a subset of $M$. The purpose of this article is the comparison
between the eigenvalues $\left(\lambda_{k}(M)\right)_{k\geq1}$ of
a Schrödinger operator $P:=-\Delta_{g}+V$ on the manifold $(M,g)$
and the eigenvalues $\left(\lambda_{k}(M-A)\right)_{k\geq1}$ of $P$
on the manifold $(M-A,g)$ with Dirichlet boundary conditions.
\end{abstract}

\section{Introduction}

The behaviour of the spectrum of a Riemannian manifold $(M,g)$ under
topological perturbation has been the subject of many research. The
most famous exemple is the crushed ice problem \textbf{{[}Kac{]}},
see also \textbf{{[}Ann{]}}. This problem consists to understand the
behaviour of Laplacian eigenvalues with Dirichlet boundary on a domain
with small holes. This subject was first studied by M. Kac \textbf{{[}Kac{]}}
in 1974. Then, J. Rauch and M. Taylor\textbf{ {[}Ra-Ta{]}} studied
the case of Euclidian Laplacian in a compact set $M$ of $\mathbb{R}^{n}$
: they showed that the spectrum of $\Delta_{\mathbb{R}^{n}}$ is invariant
by a topological excision of a $M$ by a compact subset $A$ with
a Newtonian capacity zero. Later, S. Osawa, I. Chavel and E. Feldman
\textbf{{[}Ca-Fe1{]}}, \textbf{{[}Ca-Fe2{]}} treated the Riemmannian
manifold case. They used complex probalistic techniques based on Brownian
motion. In \textbf{{[}Ge-Zh{]}}, F. Gesztesy and Z. Zhao investigate
the study the case of a Schrödinger operator with Dirichlet boundary
conditions $\mathbb{R}^{n}$, they use probabilistic tools. In 1995,
in a nice article \textbf{{[}Cou{]}} G . Courtois studied the case
of Laplace Beltrami operator on closed Riemannian manifold. He used
very simple techniques of analysis. In \textbf{{[}Be-Co{]}} J. Bertrand
and B. Colbois explained also the case of Laplace Beltrami operator
on compact Riemannian manifold. In this article we focus on the the
Schrödinger operator $-\Delta_{g}+V$ case on a closed Riemannian
manifold. 
\begin{assumption*}
The manifold is closed (i.e. compact without boundary); the function
$V$ is bounded on the manifold $M$ and $\min_{M}V>0$.
\end{assumption*}
In this work we show that under {}``little'' topological excision
of a part $A$ from the manifold, the spectrum of $-\Delta_{g}+V$
on $M-A$ is close of the spectrum on $M$. More precisely, the {}``good''
parameter for measuring the littleness of $A$ is a type of electrostatic
capacity defined by : 

\[
\textrm{cap}(A):=\textrm{inf}\left\{ Q(u),\, u\in H^{1}(M),\,\int_{M}u\, d\mathcal{V}_{g}=0,\, u-e_{1}\in H_{0}^{1}(M-A)\right\} \]
where $e_{1}$ denotes the first eigenfunction of the operator $-\Delta_{g}+V$
on the manifold $M$, and $Q$ is the following quadratic form :

\[
Q(\varphi):=\int_{M}\left|d\varphi\right|^{2}\, d\mathcal{V}_{g}+\int_{M}V\left|\varphi\right|^{2}\, d\mathcal{V}_{g}\]
and $H_{0}^{1}(M-A)$ is the \textit{Sobolev space} defined by : \[
H_{0}^{1}(M-A):=\overline{\left\{ g\in H^{1}(M),\, g=0\;\textrm{on a open neighborhood of }A\,\right\} }\]
the closure is for the norm $\left\Vert .\right\Vert _{H^{1}(M)}$,
$H^{1}(M)$ is the usual Sobolev space on $M$.\\
Indeed, more $\textrm{cap}(A)$ is small, more the spectrum $-\Delta_{g}+V$
on $M-A$ is close of the spectrum on $M$ in the following sense
:
\begin{thm*}
Let $(M,g)$ a closed Riemannian manifold. For all integer $k\geq1$,
there exists a constant $C_{k}$ depending on the manifold $(M,g)$
and on the potential $V$ such that for all subset $A$ of $M$ we
have : \[
0\leq\lambda_{k}(M-A)-\lambda_{k}(M)\leq C_{k}\sqrt{\textrm{cap}(A)}.\]

\end{thm*}
The organization of this paper is the following : in the part 2 we
start by recall some classicals results in spectral theory and about
usual Sobolev spaces, next we define our specific Sobolev space $H_{0}^{1}(M-A)$
and the notion of Schrödinger capacity. In particular, we explain
the link between the functionnal Hilbert space $H_{0}^{1}(M-A)$ and
Schrödinger capacity $\textrm{cap}(A)$. The last part of this paper
is a detailed proof of the main theorem.

\section{Spectral problem background}

\subsection{Schrödinger operator on a Riemannian manifold}

We recall here some generality on spectral geometry. In Riemannian
geometry, the \textit{Laplace Beltrami operator} is the generalisation
of Laplacian $\Delta={\displaystyle \sum_{j=1}^{n}}\frac{\partial^{2}}{\partial x_{j}^{2}}$
on $\mathbb{R}^{n}$. For a $\mathcal{C}^{2}$ real valued function
$f$ on a Riemannian manifold and for a local chart $\phi\,:\, U\subset M\rightarrow\mathbb{R}$
of the manifold $M$, the Laplace Beltrami operator is given by the
local expression :

\textit{\begin{equation}
\Delta_{g}f=\frac{1}{\sqrt{g}}{\displaystyle \sum_{j,k=1}^{n}}\frac{\partial}{\partial x_{j}}\left(\sqrt{g}g^{jk}\frac{\partial(f\circ\phi^{-1})}{\partial x_{k}}\right)\label{eq:}\end{equation}
}where $g=\det(g_{ij})$ and $g^{jk}=(g_{jk})^{-1}$.\\
The spectrum of this operator is a nice geometric invariant, see
Berger, Gauduchon and Mazet \textbf{{[}BGM{]}} and \textbf{{[}Bé-Be{]}}.
The spectrum of Laplace Beltrami operator has many applications in
geometry topology, physics ,etc ... \\
For every Riemannian manifold $(M,g)$ with dimension $n\geq1$
we have the {}``natural'' Hilbert space $L^{2}(M)=L^{2}(M,d\mathcal{V}_{g})$,
$\mathcal{V}_{g}$ is the Riemannian volume form associated to the
metric $g$. For $V$ a function from $M$ to $\mathbb{R},$ we define
the Schrödinger operator on the manifold $(M,g)$ by the linear unbounded
operator on the set of smooth compact supports real valued functions
$\mathcal{C}_{c}^{\infty}(M)\subset L^{2}(M)$ by : $-\Delta_{g}+V.$

\subsection{Sobolev spaces}

Let us denotes by $\mathcal{C}_{c}^{\infty}(M)$ the set of smooth
functions with compact support in $M$. The set $\mathcal{C}_{c}^{\infty}(M)$
is also called the set of \textit{test functions} in the language
of distributions. Recall first that the \textit{Lebesgue space} $L^{2}(M)$
on the manifold $(M,g)$ is defined by :\[
L^{2}(M):=\left\{ f\,:\, M\rightarrow\mathbb{R}\,\textrm{ measurable such that }\int_{M}\left|f\right|^{2}\, d\mathcal{V}_{g}<+\infty\right\} .\]
This space is a Hilbert space for the scalar product :

\[
\left\langle u,v\right\rangle _{L^{2}}:=\int_{M}uv\, d\mathcal{V}_{g}.\]
Next the \textit{Sobolev space} $H^{1}(M)$ is defined by : \[
H^{1}(M):=\overline{\mathcal{C}^{\infty}(M)}\]
where the closure is for the norm $\left\Vert .\right\Vert _{H^{1}}$
: $\left\Vert u\right\Vert _{H^{1}}:=\sqrt{\left\Vert u\right\Vert _{L^{2}}^{2}+\left\Vert du\right\Vert _{L^{2}}^{2}}.$
\\
An other point of view to define the space $H^{1}(M)$ is the following
:

\[
H^{1}(M)=\left\{ u\in L^{2}(M);\, du\in L^{2}(M)\right\} \]
where the derivation is the sense of distribution.\\
The space $H^{1}(M)$ is a Hilbert space for the scalar product
:

\[
\left\langle u,v\right\rangle _{H^{1}}:=\left\langle u,v\right\rangle _{L^{2}}+\left\langle du,dv\right\rangle _{L^{2}}.\]
For finish, the Sobolev space $H_{0}^{1}(M,g)$ is defined by : 

\[
H_{0}^{1}(M):=\overline{\mathcal{C}_{c}^{\infty}(M)}\]
the closure is for the norm $\left\Vert .\right\Vert _{H^{1}(M)}$
.\\
So we have :

\[
\mathcal{C}_{c}^{\infty}(M)\subset H_{0}^{1}(M)\subset H^{1}(M)\subset L^{2}(M).\]
Recall that, for the norm $\left\Vert .\right\Vert _{L^{2}(M)}$ we
have :\[
\overline{\mathcal{C}_{c}^{\infty}(M)}=L^{2}(M).\]

\subsection{Spectral problem}

The spectral problem is the following : find all pairs $(\lambda,u)$
with $\lambda\in\mathbb{R}$ and $u\in L^{2}(M)$ such that : 

\textit{\begin{equation}
-\Delta_{g}u+Vu=\lambda u\label{eq:-1}\end{equation}
}

\begin{center}
(with $u\in L^{2}(M)$ in the non-compact case).
\par\end{center}

\vspace{0.5cm}

In the case of manifold with boundary, we need boundary conditions
on the functions $u$, for example the Dirichlet conditions : $u=0$
on the boundary of $M$, or Neumann conditions : $\frac{\partial u}{\partial n}=0$
on the boundary of $M$. In the case of closed manifolds (compact
without boundary) we don't have conditions. \\
For our context (the closed case) the natural space to look here
is the Sobolev space $H^{1}(M)$.\\
Recall here a classical theorem of spectral theory (see for example\textbf{
{[}Re-Si{]}}) :
\begin{thm*}
For the above problems, the operator $-\Delta_{g}+V$ is self-adjoint,
the spectrum of the operator $-\Delta_{g}+V$ consists of a sequence
of infinite increasing eigenvalues with finite multiplicity : \[
\lambda_{1}(M)\leq\lambda_{2}(M)\leq\cdots\leq\lambda_{k}(M)\leq\cdots\rightarrow+\infty.\]
Moreover, the associate eigenfunctions $\left(e_{k}\right)_{k\geq0}$
is a Hilbert basis of the space $L^{2}(M)$.\end{thm*}
\begin{defn*}
We define the quadradic form $Q$ with domain $D(Q):=H^{1}(M)$ by
:

\[
Q(\varphi):=\int_{M}\left|d\varphi\right|^{2}\, d\mathcal{V}_{g}+\int_{M}V\left|\varphi\right|^{2}\, d\mathcal{V}_{g}.\]

\end{defn*}
Recall also (see for example \textbf{{[}Co-Hi{]}}) the minimax variational
characterization for eigenvalues : for all $k\geq1$ 

\textit{\begin{equation}
\lambda_{k}(M)=\underset{\underset{\dim(E)=k}{E\subset H^{1}(M)}}{\mathcal{\textrm{min}}}\underset{\underset{\varphi\neq0}{\varphi\in E}}{\mathcal{\textrm{max}}}R(\varphi)\label{eq:-1-1-1}\end{equation}
}where $R(\varphi)$ is the \textit{Rayleigh quotient} of the function
$\varphi$ :

\textit{\begin{equation}
R(\varphi):=\frac{Q(\varphi)}{\int_{M}\varphi^{2}\, d\mathcal{V}_{g}}.\label{eq:-1-1}\end{equation}
}In our context, a consequence of the minimax principle is :
\begin{prop*}
The first eigenvalue $\lambda_{1}(M)$ and $e_{1}$ the first eigenfunction
of the operator $-\Delta_{g}+V$ on the manifold $(M,g$) satisfy
$\lambda_{1}(M)\geq\min_{M}V>0$ and $e_{1}>0\textrm{ or }e_{1}<0\textrm{ in }M.$ \end{prop*}
\begin{proof}
It's clear that\[
\int_{M}\left|de_{1}\right|^{2}\, d\mathcal{V}_{g}+\int_{M}V\left|e_{1}\right|^{2}\, d\mathcal{V}_{g}\geq\underset{M}{\mathcal{\textrm{min}}}V\left\Vert e_{1}\right\Vert _{L^{2}(M)}^{2}\]
and on the other hand \[
\int_{M}\left|de_{1}\right|^{2}\, d\mathcal{V}_{g}+\int_{M}V\left|e_{1}\right|^{2}\, d\mathcal{V}_{g}=-\int_{M}\Delta_{g}e_{1}e_{1}\, d\mathcal{V}_{g}+\int_{M}V\left|e_{1}\right|^{2}\, d\mathcal{V}_{g}\]

\[
=\int_{M}\left(-\Delta_{g}+V\right)e_{1}e_{1}\, d\mathcal{V}_{g}=\lambda_{1}(M)\left\Vert e_{1}\right\Vert _{L^{2}(M)}^{2}\]
so $\lambda_{1}(M)\geq\min_{M}V.$ Next, suppose the function $e_{1}$
changes sign into $M$, since $e_{1}\in H^{1}(M)$, the function $f:=\left|e_{1}\right|$
belongs to $H^{1}(M)$ and $\left|df\right|=\left|de_{1}\right|$
(see for example\textbf{ {[}Gi-Tr{]}}), hence $R(f)=R\left(e_{1}\right)$.
So, the function $f$ is a first eigenfunction of $-\Delta_{g}+V$
on the manifold $M$ which satisfies $f\geq0$ on $M$, $f$ vanish
into $M$ and $\left(-\Delta_{g}+V\right)f=\lambda_{1}(M)f\geq0$
on $M$. Using the maximum principle \textbf{{[}Pr-We{]}}, the function
$f$ can not achieved it minimum in an interior point of the manifold
$M$, hence $f$ does not vanish on $M$, so we obtain a contradiction.
\end{proof}

\section{Proof of the main theorem }

\subsection{Somes other usefull spaces}

We define on the space $H^{1}(M)$ the $\star$-norm by : \[
\left\Vert u\right\Vert _{\star}^{2}:=\int_{M}\left|du\right|^{2}\, d\mathcal{V}_{g}+\int_{M}V\left|u\right|^{2}\, d\mathcal{V}_{g}\]
so, without difficulty we have :
\begin{prop*}
The application $\left\Vert .\right\Vert _{\star}$ is a norm on the
space $H^{1}(M)$; moreover this norm is equivalent to the Sobolev
norm $\left\Vert .\right\Vert _{H^{1}(M)}$. In particular $H^{1}(M),\left\Vert .\right\Vert _{\star}$
is a Banach space.
\end{prop*}
Let us denotes by $\mathcal{C}_{c}^{\infty}(M-A)$ the set of smooth
functions with compact support on $M-A$. For a compact subset $A$
of the manifold $M$ the usual Sobolev space $H_{0}^{1}(M-A)$ is
defined by the closure of $\mathcal{C}_{c}^{\infty}(M-A)$ for the
norm $\left\Vert .\right\Vert _{H^{1}(M)}$ :\[
H_{0}^{1}(M-A):=\overline{\mathcal{C}_{c}^{\infty}(M-A)}.\]
What happens when the set $A$ is not compact ? For example if $A$
is a dense and countable subset of points of the manifold $M$, the
space of test functions $\mathcal{C}_{c}^{\infty}(M-A)$ is reduced
to $\{0\}$. Therefore we cannot define the space $H_{0}^{1}(M-A)$.
In this case, we propose a definition of $H_{0}^{1}(M-A)$ for any
subset $A$ of $M$.
\begin{defn*}
We define the Sobolev spaces $\mathcal{H}_{0}^{1}(M-A)$ and $H_{0}^{1}(M-A)$
by :\[
\mathcal{H}_{0}^{1}(M-A):=\left\{ g\in H^{1}(M),\, g=0\;\textrm{on a open neighborhood of }A\,\right\} ;\]
\[
H_{0}^{1}(M-A):=\overline{\mathcal{H}_{0}^{1}(M-A)}\]
where the closure is for the norm $\left\Vert .\right\Vert _{H^{1}(M)}.$
\end{defn*}
We have the :
\begin{prop*}
If the set $A$ is compact, the previous definition of the space $H_{0}^{1}(M-A)$
coincides with the usal ones.\end{prop*}
\begin{proof}
Let $f\in H_{0}^{1}(M-A):=\overline{\mathcal{H}_{0}^{1}(M-A)}$, then
by definition : for all $\varepsilon\geq0$ there exists $g\in\mathcal{H}_{0}^{1}(M-A)$
such that $\left\Vert f-g\right\Vert _{H^{1}(M)}\leq\varepsilon$.
So, we will show that we can write $g$ as a limit of sequence from
the space $\mathcal{C}_{c}^{\infty}(M-A)$ and conclude. Since $g\in\mathcal{H}_{0}^{1}(M-A)$
there exists an open set $U\supset A$ such that $g_{|U}=0$. Consider
two open sets $U_{1}$ and $U_{2}$ of the manifold $M$ such that
:\[
A\subset U_{1},\, M-U\subset U_{2},\, U_{1}\cap U_{2}=\emptyset;\]
and consider also a function $\varphi\in\mathcal{D}(M)$ such that
:\[
\varphi_{|U_{1}}=0,\;\varphi_{|U_{2}}=1.\]
Of course, the function $\varphi$ belongs to the space $\mathcal{C}_{c}^{\infty}(M-A)$.
Next, since $g\in\mathcal{H}_{0}^{1}(M-A)\subset H^{1}(M)$ and as
the set of smooth functions $\mathcal{C}^{\infty}(M)$ is dense in
$H^{1}(M)$ : there exists a sequence $\left(g_{n}\right)_{n}$ in
$\mathcal{C}^{\infty}(M)$ such that ${\displaystyle \lim_{n\rightarrow+\infty}}g_{n}=g$
for the norm $\left\Vert .\right\Vert _{H^{1}(M)}.$ Therefore we
claim that : ${\displaystyle \lim_{n\rightarrow+\infty}}\varphi g_{n}=g$
for the norm $\left\Vert .\right\Vert _{H^{1}(M)}$. Indeed, start
by, for all integer $n$ : \[
\left\Vert \varphi g_{n}-g\right\Vert _{H^{1}(M)}^{2}\leq\left\Vert g_{n}-g\right\Vert _{H^{1}(M-U)}^{2}+\left\Vert \varphi g_{n}-g\right\Vert _{H^{1}(U)}^{2}\]
\[
\leq\left\Vert g_{n}-g\right\Vert _{H^{1}(M)}^{2}+\left\Vert \varphi g_{n}-g\right\Vert _{H^{1}(U)}^{2}.\]
Next, we observe that, for all integer $n$ :\[
\left\Vert \varphi g_{n}-g\right\Vert _{H^{1}(U)}^{2}=\left\Vert \varphi g_{n}\right\Vert _{H^{1}(U)}^{2}\]
\[
=\int_{U}\left|\varphi g_{n}\right|^{2}\, d\mathcal{V}_{g}+\int_{U}\left|d\varphi g_{n}+\varphi dg_{n}\right|^{2}\, d\mathcal{V}_{g}\]
\[
\leq\int_{U}\left|\varphi g_{n}\right|^{2}\, d\mathcal{V}_{g}+\int_{U}\left|d\varphi g_{n}\right|^{2}\, d\mathcal{V}_{g}+\int_{U}\left|\varphi dg_{n}\right|^{2}\, d\mathcal{V}_{g}+2\int_{U}\left|d\varphi g_{n}\varphi dg_{n}\right|\, d\mathcal{V}_{g}\]
\[
\leq\left\Vert \varphi\right\Vert _{\infty}^{2}\left\Vert g_{n}\right\Vert _{L^{2}(U)}^{2}+\left\Vert d\varphi\right\Vert _{L^{\infty}(M)}^{2}\left\Vert g_{n}\right\Vert _{L^{2}(U)}^{2}\]
\[
+\left\Vert \varphi\right\Vert _{\infty}^{2}\left\Vert dg_{n}\right\Vert _{L^{2}(U)}^{2}+2\left\Vert d\varphi\right\Vert _{\infty}\left\Vert \varphi\right\Vert _{\infty}\int_{U}\left|g_{n}dg_{n}\right|\, d\mathcal{V}_{g}\]
\[
\leq\left\Vert \varphi\right\Vert _{\infty}^{2}\left\Vert g_{n}\right\Vert _{L^{2}(U)}^{2}+\left\Vert d\varphi\right\Vert _{\infty}^{2}\left\Vert g_{n}\right\Vert _{L^{2}(U)}^{2}\]
\[
+\left\Vert \varphi\right\Vert _{\infty}^{2}\left\Vert dg_{n}\right\Vert _{L^{2}(U)}^{2}+2\left\Vert d\varphi\right\Vert _{\infty}\left\Vert \varphi\right\Vert _{L^{\infty}(M)}\left\Vert g_{n}\right\Vert _{L^{2}(U)}\left\Vert dg_{n}\right\Vert _{L^{2}(U)},\]
by Cauchy-Schwarz inequality. \\
Finally we get for all integer $n$ :\[
\left\Vert \varphi g_{n}-g\right\Vert _{H^{1}(U)}^{2}\leq\left\Vert g_{n}\right\Vert _{H^{1}(U)}^{2}\left(2\left\Vert \varphi\right\Vert _{\infty}^{2}+\left\Vert d\varphi\right\Vert _{\infty}^{2}+2\left\Vert d\varphi\right\Vert _{\infty}\left\Vert \varphi\right\Vert _{\infty}\right).\]
As a consequence, we have for all integer $n$ :\[
\left\Vert \varphi g_{n}-g\right\Vert _{H^{1}(M)}^{2}\leq\left\Vert g_{n}-g\right\Vert _{H^{1}(M-U)}^{2}\]
\[
+\left\Vert g_{n}\right\Vert _{H^{1}(U)}^{2}\left(2\left\Vert \varphi\right\Vert _{\infty}^{2}+\left\Vert d\varphi\right\Vert _{\infty}^{2}+2\left\Vert d\varphi\right\Vert _{\infty}\left\Vert \varphi\right\Vert _{\infty}\right).\]
Now, it suffices to note that $\left\Vert g_{n}\right\Vert _{H^{1}(U)}^{2}=\left\Vert g_{n}-g\right\Vert _{H^{1}(U)}^{2}\leq\left\Vert g_{n}-g\right\Vert _{H^{1}(M)}^{2}$
(since $g=0$ on the open set $U$) and we have finally :\[
\left\Vert \varphi g_{n}-g\right\Vert _{H^{1}(M)}^{2}\leq\]
\[
\left\Vert g_{n}-g\right\Vert _{H^{1}(M)}^{2}\left(1+2\left\Vert \varphi\right\Vert _{\infty}^{2}+\left\Vert d\varphi\right\Vert _{\infty}^{2}+2\left\Vert d\varphi\right\Vert _{\infty}\left\Vert \varphi\right\Vert _{\infty}\right).\]
The sequence $\left(\varphi g_{n}\right)_{n}$ belong to $\mathcal{C}_{c}^{\infty}(M-A){}^{\mathbb{N}},$and
since ${\displaystyle \lim_{n\rightarrow+\infty}}g_{n}=g$ for the
norm $\left\Vert .\right\Vert _{H^{1}(M)}$ the previous inequality
implies ${\displaystyle \lim_{n\rightarrow+\infty}}\varphi g_{n}=g$
for the norm $\left\Vert .\right\Vert _{H^{1}(M)}.$ \\
So we have shown that every function $f\in H_{0}^{1}(M-A):=\overline{\mathcal{H}_{0}^{1}(M-A)}$
is a limit (for the norm $\left\Vert .\right\Vert _{H^{1}(M)}$) of
a sequence of $\mathcal{C}_{c}^{\infty}(M-A)$. \\
Conversely, since $\mathcal{C}_{c}^{\infty}(M-A)\subset\mathcal{H}_{0}^{1}(M-A)$
we get : \[
H_{0}^{1}(M-A):=\overline{\mathcal{C}_{c}^{\infty}(M-A)}\subset H_{0}^{1}(M-A):=\overline{\mathcal{H}_{0}^{1}(M-A)}.\]

\end{proof}
Let us also denote the spaces $H_{\star}^{1}(M)$ and $S_{A}(M)$
by :

\[
H_{\star}^{1}(M):=\left\{ f\in H^{1}(M),\,\int_{M}f\, d\mathcal{V}_{g}=0\right\} ;\]
and 

\[
S_{A}(M):=\left\{ u\in H_{\star}^{1}(M),\, u-e_{1}\in H_{0}^{1}(M-A)\right\} .\]
In the definition of the space $H_{\star}^{1}(M)$ the condition $\int_{M}f\, d\mathcal{V}_{g}=0$
is analog to a boundary condition. We observe that the space $H_{\star}^{1}(M)$
is a Hilbert space for the norm :

\[
\left\Vert u\right\Vert _{\star}:=\int_{M}\left|du\right|^{2}\, d\mathcal{V}_{g}+\int_{M}V\left|u\right|^{2}\, d\mathcal{V}_{g};\]
and $S_{A}(M)$ is just an affine closed subset of $H^{1}(M)$.

\subsection{Schrödinger capacity}

Next, we introduce the Schrödinger capacity of the set $A$ ;
\begin{defn*}
Let us consider the \textit{Schrödinger capacity} $\textrm{cap}(A)$
of the set $A$ defined by\textit{\begin{equation}
\textrm{cap}(A):=\textrm{inf}\left\{ \int_{M}\left|du\right|^{2}\, d\mathcal{V}_{g}+\int_{M}V\left|u\right|^{2}\, d\mathcal{V}_{g},\, u\in S_{A}(M)\right\} .\label{eq:-1-1-2}\end{equation}
}
\end{defn*}
Let us remark that : there exists an unique function $u_{A}\in S_{A}(M)$
such that 

\[
\textrm{cap}(A)=\int_{M}\left|du_{A}\right|^{2}\, d\mathcal{V}_{g}+\int_{M}V\left|u_{A}\right|^{2}\, d\mathcal{V}_{g}.\]
Indeed : here the capacity $\textrm{cap}(A)$ is just the distance
between the function $0$ and the closed space $S_{A}(M)$. This distance
is equal to $\left\Vert u_{A}\right\Vert _{\star}$ where $u_{A}$
is the orthogonal projection of $0$ on $S_{A}(M)$ :\textit{\[
\textrm{cap}(A)=d_{\star}\left(0,S_{A}(M)\right):=\inf\left\{ \left\Vert u\right\Vert _{\star},\, u\in S_{A}(M)\right\} =\left\Vert u_{A}\right\Vert _{\star}.\]
}In the following lemma we give the relationships between the capacity
$\textrm{cap}(A)$, the functions $u_{A},\, e_{1}$ and the Sobolev
spaces $H_{0}^{1}(M-A),\, H^{1}(M)$.
\begin{lem*}
For all subset $A$ of the manifold $M$, the following properties
are equivalent :

(i) $\textrm{cap}(A)=0$;

(ii) $u_{A}=0$;

(iii) $e_{1}\in H_{0}^{1}(M-A)$;

(iv) $H_{0}^{1}(M-A)=H^{1}(M)$.\end{lem*}
\begin{proof}
It is clear from the formula (3.1) that $(i)\Leftrightarrow(ii)\Leftrightarrow(iii)$.
Next, suppose the property $(iii)$ holds : so there exists a sequence
$(v_{n})_{n}\in\mathcal{H}_{0}^{1}(M-A)^{\mathbb{N}}$ such that ${\displaystyle \lim_{n\rightarrow+\infty}}v_{n}=e_{1}$
for the norm $\left\Vert .\right\Vert _{H^{1}(M)}.$ So, for all smooth
function $\varphi\in\mathcal{C}^{\infty}(M)$ we have ${\displaystyle \lim_{n\rightarrow+\infty}(\varphi}v_{n})/e_{1}=\varphi$
for the norm $\left\Vert .\right\Vert _{H^{1}(M)}$, indeed for all
integer $n$ :\[
\left\Vert \frac{\varphi v_{n}}{e_{1}}-\varphi\right\Vert _{H^{1}(M)}^{2}=\int_{M}\left|\frac{\varphi v_{n}}{e_{1}}-\varphi\right|^{2}\, d\mathcal{V}_{g}+\int_{M}\left|d\left(\frac{\varphi v_{n}}{e_{1}}\right)-d\varphi\right|^{2}\, d\mathcal{V}_{g}.\]
First, we have for all integer $n$ : \[
\int_{M}\left|\frac{\varphi v_{n}}{e_{1}}-\varphi\right|^{2}\, d\mathcal{V}_{g}=\int_{M}\frac{1}{\left|e_{1}\right|^{2}}\left|\varphi\left(v_{n}-e_{1}\right)\right|^{2}\, d\mathcal{V}_{g}\]
\[
\leq\left\Vert \frac{1}{e_{1}}\right\Vert _{\infty}^{2}\left\Vert \varphi\right\Vert _{\infty}^{2}\left\Vert v_{n}-e_{1}\right\Vert _{L^{^{2}}(M)}^{2}\]
so, since ${\displaystyle \lim_{n\rightarrow+\infty}}v_{n}=e_{1}$
for the norm $\left\Vert .\right\Vert _{H^{1}(M)}$ we have \[
{\displaystyle \lim_{n\rightarrow+\infty}}\int_{M}\left|\frac{\varphi v_{n}}{e_{1}}-\varphi\right|^{2}\, d\mathcal{V}_{g}=0.\]
On the other hand, for all integer $n$ :\[
\int_{M}\left|d\left(\frac{\varphi v_{n}}{e_{1}}\right)-d\varphi\right|^{2}\, d\mathcal{V}_{g}=\int_{M}\left|\frac{d\left(\varphi v_{n}\right)e_{1}-\varphi v_{n}de_{1}}{e_{1}^{2}}-d\varphi\right|^{2}\, d\mathcal{V}_{g}\]
\[
=\int_{M}\left(\frac{1}{e_{1}^{2}}\right)\left|d\left(\varphi\right)v_{n}e_{1}+\varphi d\left(v_{n}\right)e_{1}-\varphi v_{n}d\left(e_{1}\right)-d\left(\varphi\right)e_{1}^{2}\right|^{2}\, d\mathcal{V}_{g}\]
\[
\leq\left\Vert \frac{1}{e_{1}}\right\Vert _{\infty}^{2}\left\Vert d\varphi v_{n}e_{1}-d\varphi e_{1}^{2}+\varphi dv_{n}e_{1}-\varphi v_{n}de_{1}\right\Vert _{L^{^{2}}(M)}^{2}\]
\[
\leq\left\Vert \frac{1}{e_{1}}\right\Vert _{\infty}^{2}\left(\left\Vert d\varphi v_{n}e_{1}-d\varphi e_{1}^{2}\right\Vert _{L^{^{2}}(M)}+\left\Vert \varphi dv_{n}e_{1}-\varphi v_{n}de_{1}\right\Vert _{L^{^{2}}(M)}\right)^{2}\]
\[
\leq\left\Vert \frac{1}{e_{1}}\right\Vert _{\infty}^{2}\left[\left\Vert d\varphi\right\Vert _{\infty}\left\Vert e_{1}\right\Vert _{\infty}\left\Vert v_{n}-e_{1}\right\Vert _{L^{^{2}}(M)}+\right.\]
\[
\left.\left\Vert \varphi\right\Vert _{\infty}\left\Vert e_{1}\left(dv_{n}-de_{1}\right)+e_{1}de_{1}-v_{n}de_{1}\right\Vert _{L^{^{2}}(M)}\right]^{2}\]
\[
\leq\left\Vert \frac{1}{e_{1}}\right\Vert _{\infty}^{2}\left[\left\Vert d\varphi\right\Vert _{\infty}\left\Vert e_{1}\right\Vert _{\infty}\left\Vert v_{n}-e_{1}\right\Vert _{L^{^{2}}(M)}+\right.\]
\[
\left.\left\Vert \varphi\right\Vert _{\infty}\left\Vert e_{1}\right\Vert _{\infty}\left\Vert dv_{n}-de_{1}\right\Vert _{L^{^{2}}(M)}+\left\Vert \varphi\right\Vert _{\infty}\left\Vert de_{1}\right\Vert _{\infty}\left\Vert e_{1}-v_{n}\right\Vert _{L^{^{2}}(M)}\right]^{2};\]
so, since ${\displaystyle \lim_{n\rightarrow+\infty}}v_{n}=e_{1}$
for the norm $\left\Vert .\right\Vert _{H^{1}(M)}$ we have \[
{\displaystyle \lim_{n\rightarrow+\infty}}\int_{M}\left|d\left(\frac{\varphi v_{n}}{e_{1}}\right)-d\varphi\right|^{2}\, d\mathcal{V}_{g}=0.\]
Therefore, for all function $\varphi\in\mathcal{C}^{\infty}(M)$ we
have ${\displaystyle \lim_{n\rightarrow+\infty}}\frac{\varphi v_{n}}{e_{1}}=\varphi$
for the norm $\left\Vert .\right\Vert _{H^{1}(M)}$.\\
Next, by density of $\mathcal{C}^{\infty}(M)$ in $H^{1}(M)$ :
for all function $f\in H^{1}(M)$ we have ${\displaystyle \lim_{n\rightarrow+\infty}\frac{fv_{n}}{e_{1}}}=f$
. Since the sequence $\left(\frac{fv_{n}}{e_{1}}\right)_{n}\in\mathcal{H}_{0}^{1}(M-A)^{\mathbb{N}}$
we get finally that $f$ belongs to space $H_{0}^{1}(M-A)$. Finally,
it is easy to see that $(iv)\Rightarrow(iii).$
\end{proof}
An obvious consequence of this lemma is the following result :
\begin{prop*}
The spectrum of $-\Delta_{g}+V$ on the manifold $(M,g)$ and on the
manifold $(M-A,g)$ are equal if and only if $\textrm{cap}(A)=0$.
\end{prop*}

\subsection{The Poincaré inequality}

Now, let introduce the Poincaré inequality :
\begin{thm*}
If $\lambda_{1}(M)$ denotes the first eigenvalue of the operator
$-\Delta_{g}+V$ on the manifold $(M,g)$, the following inequality\begin{equation}
\left\Vert u_{A}\right\Vert _{L^{2}(M)}^{2}\leq\frac{\textrm{cap}(A)}{\lambda_{1}(M)}\label{eq:-1-1-3}\end{equation}
holds for all subset $A$ of $M$.\end{thm*}
\begin{proof}
The case $\textrm{cap}(A)=0$ is an obvious consequence of the lemma
in section 3.2. Suppose here that $\textrm{cap}(A)>0$, then $\left\Vert u_{A}\right\Vert _{L^{2}(M)}>0$.
The first eigenvalue $\lambda_{1}(M)$ of the operator $-\Delta_{g}+V$
on the manifold $(M,g)$ is given by :\[
\lambda_{1}(M)=\underset{\underset{\dim(E)=1}{E\subset H^{1}(M)}}{\mathcal{\textrm{min}}}\underset{\underset{\varphi\neq0}{\varphi\in E}}{\mathcal{\textrm{max}}}\frac{\int_{M}\left|d\varphi\right|^{2}+V\left|\varphi\right|^{2}\, d\mathcal{V}_{g}}{\int_{M}\left|\varphi\right|^{2}\, d\mathcal{V}_{g}}\]
\[
=\underset{\underset{\varphi\neq0}{\varphi\in H^{1}(M)}}{\mathcal{\textrm{min}}}\frac{\int_{M}\left|d\varphi\right|^{2}+V\left|\varphi\right|^{2}\, d\mathcal{V}_{g}}{\int_{M}\left|\varphi\right|^{2}\, d\mathcal{V}_{g}}\]
Since $u_{A}$ belongs to the space $H^{1}(M)$ we get $\lambda_{1}(M)\leq\frac{\textrm{cap}(A)}{\left\Vert u_{A}\right\Vert _{L^{2}(M)}^{2}}.$
\end{proof}

\subsection{The main theorem}

Recall our main result : 
\begin{thm*}
Let $(M,g)$ a compact Riemannian manifold. For all integer $k\geq1$,
there exists a constant $C_{k}$ depending on the manifold of $(M,g)$
and the potential $V$ such that for all subset $A$ of $M$ we have
:\[
0\leq\lambda_{k}(M-A)-\lambda_{k}(M)\leq C_{k}\sqrt{\textrm{cap}(A)}.\]
\end{thm*}
\begin{rem*}
We can easily adapt the proof for a compact Riemannian manifold with
boundary.\end{rem*}
\begin{proof}
Let us denote by $\left(e_{k}\right)_{k\geq1}$ an orthonormal basis
of the space $L^{2}(M)$ with eigenfunctions of the operator $-\Delta_{g}+V$
on the manifold $(M,g$). For all integer $k\geq1,$ we consider the
sets \[
F_{k}:=\textrm{span}\left\{ e_{1},e_{2},\ldots,e_{k}\right\} \]
and\[
E_{k}:=\left\{ f\left(1-\frac{u_{A}}{e_{1}}\right),\, f\in F_{k}\right\} .\]
First, observe that $E_{k}\subset H_{0}^{1}(M-A)$. For all $j\in\{1,\ldots,k\}$
we introduce also the functions $\phi_{j}:=e_{j}\left(1-\frac{u_{A}}{e_{1}}\right)\in E_{k}$.\\
\\
$\bullet$\textbf{ Step 1} : we compute the $L^{2}$-inner product
$\left\langle \phi_{i},\phi_{j}\right\rangle _{L^{2}(M)}$ for all
pairs $(i,j)\in\{1,\ldots,k\}^{2}$ :\[
\left\langle \phi_{i},\phi_{j}\right\rangle _{L^{2}(M)}=\int_{M}e_{i}e_{j}\left(1-\frac{u_{A}}{e_{1}}\right)^{2}\, d\mathcal{V}_{g}\]
\[
=\delta_{i,j}-2\int_{M}\frac{e_{i}e_{j}}{e_{1}}u_{A}\, d\mathcal{V}_{g}+\int_{M}e_{i}e_{j}\frac{u_{A}^{2}}{e_{1}^{2}}\, d\mathcal{V}_{g}.\]
Thus, for all pair $(i,j)\in\{1,\ldots,k\}^{2}$ we get :\[
\left|\left\langle \phi_{i},\phi_{j}\right\rangle _{L^{2}(M)}-\delta_{i,j}\right|\leq2\int_{M}\left|\frac{e_{i}e_{j}}{e_{1}}u_{A}\right|\, d\mathcal{V}_{g}+\int_{M}\left|e_{i}e_{j}\frac{u_{A}^{2}}{e_{1}^{2}}\right|\, d\mathcal{V}_{g},\]
hence, by Cauchy-Schwarz inequality we obtain \[
\left|\left\langle \phi_{i},\phi_{j}\right\rangle _{L^{2}(M)}-\delta_{i,j}\right|\leq2{\displaystyle \max_{1\leq i,j\leq k}\left\Vert \frac{e_{i}e_{j}}{e_{1}^{2}}\right\Vert _{\infty}}\left\Vert u_{A}\right\Vert _{L^{2}(M)}+{\displaystyle \max_{1\leq i,j\leq k}\left\Vert \frac{e_{i}e_{j}}{e_{1}^{2}}\right\Vert _{\infty}}\left\Vert u_{A}\right\Vert _{L^{2}(M)}^{2}\]
\[
\leq2\max_{1\leq i,j\leq k}\left\Vert \frac{e_{i}e_{j}}{e_{1}}\right\Vert _{\infty}\sqrt{\textrm{vol}(M)}\left\Vert u_{A}\right\Vert _{L^{2}(M)}+\max_{1\leq i,j\leq k}\left\Vert \frac{e_{i}e_{j}}{e_{1}^{2}}\right\Vert _{\infty}\left\Vert u_{A}\right\Vert _{L^{2}(M)}^{2}\]
hence by Poincaré inequality we have \[
\left|\left\langle \phi_{i},\phi_{j}\right\rangle _{L^{2}(M)}-\delta_{i,j}\right|\leq B_{k,M}\left(\sqrt{\textrm{cap}(A)}+\textrm{cap}(A)\right)\]
where $B_{k}=B_{k}\left(e_{1},e_{2},...,e_{k},\lambda_{1}(M),M\right)\geq0$,
and since the eigenfunctions $e_{1},e_{2},...,e_{k}$ and the eigenvalue
$\lambda_{1}(M)$ depends only on $(M,g)$ and $V$, for all integer
$k$ the constant $B_{k}$ depends only on $(M,g)$ and $V$, ie :
$B_{k}=B_{k}\left(M,V\right).$ \\
Therefore, there exists $\varepsilon_{k}\in]0,1[$ (depends on
the constant \foreignlanguage{english}{$B_{k}$}) such that for all
$A\subset M$ we have :\[
\textrm{cap}(A)\leq\varepsilon_{k}\Rightarrow\textrm{dim}(E_{k})=k\;\textrm{and}\;\forall j\in\{1,...,k\},\,\left|\left\Vert \phi_{j}\right\Vert _{L^{2}(M)}^{2}-1\right|\leq D_{k}\sqrt{\textrm{cap}(A)}\]
where (and for the same reasons as in the study of $B_{k}$) for all
integer $k$, the constant $D_{k}$ depends only on $M$ and $V$,
ie $D_{k}=D_{k}\left(M,V\right)$.\\
\\
$\bullet$\textbf{ Step 2} : Let a function $\phi=f\left(1-\frac{u_{A}}{e_{1}}\right)\in E_{k}$,
with $f\in F_{k}$. Without loss generality we can assume that $\left\Vert f\right\Vert _{L^{2}(M)}=1$,
indeed : we have $R(\phi)=R\left(\frac{\phi}{\left\Vert f\right\Vert _{L^{2}(M)}}\right)$
and in our context we intererest in the Rayleigh quotient of $\phi$
(see the end of the final step of the proof). \\
Set $v_{A}:=\frac{u_{A}}{e_{1}}$, we have :\[
\int_{M}\left|d\phi\right|^{2}d\mathcal{V}_{g}=\int_{M}\left|df-d\left(fv_{A}\right)\right|^{2}\, d\mathcal{V}_{g}\]
\[
=\int_{M}\left|df\right|^{2}\, d\mathcal{V}_{g}+\int_{M}\left|dfv_{A}+fdv_{A}\right|^{2}\, d\mathcal{V}_{g}-2\int_{M}dfd\left(fv_{A}\right)\, d\mathcal{V}_{g}\]
\[
=\int_{M}\left|df\right|^{2}\, d\mathcal{V}_{g}+\int_{M}\left|dfv_{A}\right|^{2}\, d\mathcal{V}_{g}+\int_{M}\left|fdv_{A}\right|^{2}\, d\mathcal{V}_{g}\]
\[
+2\int_{M}dfdv_{A}fv_{A}\, d\mathcal{V}_{g}-2\int_{M}\left|df\right|^{2}v_{A}\, d\mathcal{V}_{g}-2\int_{M}dfdv_{A}f\, d\mathcal{V}_{g}\]
\[
=\int_{M}\left|df\right|^{2}\, d\mathcal{V}_{g}+\int_{M}\left|dfv_{A}\right|^{2}\, d\mathcal{V}_{g}+\int_{M}\left|fdv_{A}\right|^{2}\, d\mathcal{V}_{g}\]
\[
-2\int_{M}\left|df\right|^{2}v_{A}\, d\mathcal{V}_{g}-2\int_{M}dfdv_{A}f\left(1-v_{A}\right)\, d\mathcal{V}_{g}.\]
Recall we have $dv_{A}=\frac{du_{A}e_{1}-u_{A}de_{1}}{e_{1}^{2}}$,
and :\[
\int_{M}V\left|\phi\right|^{2}d\mathcal{V}_{g}=\int_{M}V\left|f\right|^{2}\, d\mathcal{V}_{g}-2\int_{M}V\left|f\right|^{2}v_{A}\, d\mathcal{V}_{g}+\int_{M}V\left|v_{A}f\right|^{2}\, d\mathcal{V}_{g}\]
hence\[
\int_{M}\left|d\phi\right|^{2}d\mathcal{V}_{g}+\int_{M}V\left|\phi\right|^{2}d\mathcal{V}_{g}=\underset{:=A(f)}{\underbrace{\int_{M}\left|df\right|^{2}\, d\mathcal{V}_{g}+\int_{M}V\left|f\right|^{2}\, d\mathcal{V}_{g}}}+\underset{:=B(f)}{\underbrace{\int_{M}\left|dfv_{A}\right|^{2}\, d\mathcal{V}_{g}}}\]
\[
+\underset{:=C(f)}{\underbrace{\int_{M}\left|fdv_{A}\right|^{2}\, d\mathcal{V}_{g}+\int_{M}V\left|v_{A}f\right|^{2}\, d\mathcal{V}_{g}}}-2\left(\underset{:=D(f)}{\underbrace{\int_{M}\left|df\right|^{2}v_{A}\, d\mathcal{V}_{g}+\int_{M}V\left|f\right|^{2}v_{A}\, d\mathcal{V}_{g}}}\right)\]
\[
-2\underset{:=E(f)}{\underbrace{\int_{M}dfdv_{A}f\left(1-v_{A}\right)\, d\mathcal{V}_{g}}}.\]
$\blacklozenge$ Study of $A(f):=\int_{M}\left|df\right|^{2}\, d\mathcal{V}_{g}+\int_{M}V\left|f\right|^{2}\, d\mathcal{V}_{g}\geq0$
: since $f\in F_{k}$ we can write $f={\displaystyle \sum_{i=1}^{k}\alpha_{i}e_{i}}$
where $\left(\alpha_{i}\right)_{1\leq i\leq k}\in\mathbb{R}^{k}$
and with ${\displaystyle \sum_{i=1}^{k}\alpha_{i}^{2}}=1$ (since
$\left\Vert f\right\Vert _{L^{2}(M)}=1$), thus we get\[
A(f)=\left\langle {\displaystyle \sum_{j=1}^{k}\alpha_{j}de_{j}},\sum_{i=1}^{k}\alpha_{i}de_{i}\right\rangle _{L^{2}(M)}+\left\langle \sqrt{V}{\displaystyle \sum_{j=1}^{k}\alpha_{j}e_{j}},\sqrt{V}\sum_{i=1}^{k}\alpha_{i}e_{i}\right\rangle _{L^{2}(M)}\]
\[
={\displaystyle \sum_{i,j}\alpha_{i}\alpha_{j}}\left(\left\langle {\displaystyle de_{j}},de_{i}\right\rangle _{L^{2}(M)}+\int_{M}Ve_{j}e_{i}\, d\mathcal{V}_{g}\right)\]
\[
={\displaystyle \sum_{i,j}\alpha_{i}\alpha_{j}}\left(-\left\langle {\displaystyle e_{j}},\Delta_{g}e_{i}\right\rangle _{L^{2}(M)}+\int_{M}Ve_{j}e_{i}\, d\mathcal{V}_{g}\right)\]
\[
={\displaystyle \sum_{i,j}\alpha_{i}\alpha_{j}}\left\langle {\displaystyle e_{j}},\left(-\Delta_{g}+V\right)e_{i}\right\rangle _{L^{2}(M)}\]
\[
={\displaystyle \sum_{i,j}\alpha_{i}\alpha_{j}}\lambda_{i}(M)\left\langle {\displaystyle e_{j}},e_{i}\right\rangle _{L^{2}(M)}={\displaystyle \sum_{i=1}^{k}\alpha_{i}^{2}}\lambda_{i}(M)\leq\lambda_{k}(M).\]
Hence, for all integer $k$, and for all function $f\in F_{k}$ such
that $\left\Vert f\right\Vert _{L^{2}(M)}=1$ we have

\textit{\begin{equation}
0\leq A(f)\leq\lambda_{k}(M).\label{eq:-1-1-3-1}\end{equation}
}$\blacklozenge$ Study of $B(f):=\int_{M}\left|d(f)v_{A}\right|^{2}\, d\mathcal{V}_{g}$
: here $v_{A}=\frac{u_{A}}{e_{1}}$ and $dv_{A}=\frac{du_{A}e_{1}-u_{A}de_{1}}{e_{1}^{2}}$,
so we get $B\leq\left\Vert df\right\Vert _{\infty}^{2}\left\Vert v_{A}\right\Vert _{L^{2}(M)}^{2}$
and, with the Poincaré inequality :\[
\left\Vert v_{A}\right\Vert _{L^{2}(M)}^{2}\leq\left\Vert \frac{1}{e_{1}}\right\Vert _{\infty}^{2}\left\Vert u_{A}\right\Vert _{L^{2}(M)}^{2}\leq\left\Vert \frac{1}{e_{1}}\right\Vert _{\infty}^{2}\frac{\textrm{cap}(A)}{\lambda_{1}(M)}\]
hence, for all integer $k$, and for all function $f\in F_{k}$ such
that $\left\Vert f\right\Vert _{L^{2}(M)}=1$ we have

\textit{\begin{equation}
0\leq B(f)\leq E_{k}\textrm{cap}(A)\label{eq:-1-1-3-1-1}\end{equation}
}where $E_{k}=E_{k}\left(e_{1},\lambda_{1}(M)\right)>0$, moreover
since the eigenfunction $e_{1}$ and the eigenvalue $\lambda_{1}(M)$
depends only on $(M,g)$ and $V$, for all integer $k$ the constant
$E_{k}$ depends only on $(M,g)$ and $V$, ie : $E_{k}=E_{k}\left(M,V\right).$
\\
$\blacklozenge$ Study of $C(f)$ : here $C(f)$ is equal to $\underset{:=C_{1}(f)}{\underbrace{\int_{M}\left|fdv_{A}\right|^{2}\, d\mathcal{V}_{g}}+}\underset{:=C_{2}(f)}{\underbrace{\int_{M}V\left|v_{A}f\right|^{2}\, d\mathcal{V}_{g}}}$.
Let us observe first $C_{1}(f)$ : \[
C_{1}(f)\leq\left\Vert f\right\Vert _{\infty}^{2}\left\Vert dv_{A}\right\Vert _{L^{2}(M)}^{2}\]
and\[
\left\Vert dv_{A}\right\Vert _{L^{2}(M)}^{2}=\int_{M}\left|\frac{du_{A}e_{1}-u_{A}de_{1}}{e_{1}^{2}}\right|^{2}\, d\mathcal{V}_{g}\]
\[
\leq\left\Vert \frac{1}{e_{1}}\right\Vert _{\infty}^{2}\int_{M}\left|du_{A}e_{1}-u_{A}de_{1}\right|^{2}\, d\mathcal{V}_{g}\]
\[
\leq\left\Vert \frac{1}{e_{1}}\right\Vert _{\infty}^{2}\left(\int_{M}\left|du_{A}e_{1}\right|^{2}\, d\mathcal{V}_{g}+2\int_{M}\left|du_{A}de_{1}e_{1}u_{A}\right|\, d\mathcal{V}_{g}+\int_{M}\left|de_{1}u_{A}\right|^{2}\, d\mathcal{V}_{g}\right)\]
\[
\leq\left\Vert \frac{1}{e_{1}}\right\Vert _{\infty}^{2}\left(\left\Vert du_{A}\right\Vert _{L^{2}(M)}^{2}\left\Vert e_{1}\right\Vert _{\infty}^{2}+2\left\Vert de_{1}\right\Vert _{\infty}\left\Vert e_{1}\right\Vert _{\infty}\left\Vert du_{A}\right\Vert _{L^{2}(M)}\left\Vert u_{A}\right\Vert _{L^{2}(M)}+\left\Vert de_{1}\right\Vert _{\infty}^{2}\left\Vert u_{A}\right\Vert _{L^{2}(M)}^{2}\right).\]
Next we have also :\[
C_{2}(f)=\int_{M}V\left|v_{A}f\right|^{2}\, d\mathcal{V}_{g}\leq\left\Vert f\right\Vert _{\infty}^{2}\int_{M}V\left|v_{A}\right|^{2}\, d\mathcal{V}_{g}\]
\[
\leq\left\Vert f\right\Vert _{\infty}^{2}\left\Vert \frac{1}{e_{1}}\right\Vert _{\infty}^{2}\int_{M}V\left|u_{A}\right|^{2}\, d\mathcal{V}_{g}.\]
Hence we get :

\[
C(f)\leq\left\Vert f\right\Vert _{\infty}^{2}\left\Vert \frac{1}{e_{1}}\right\Vert _{\infty}^{2}\left[\left\Vert du_{A}\right\Vert _{L^{2}(M)}^{2}\left\Vert e_{1}\right\Vert _{\infty}^{2}\right.\]
\[
\left.+2\left\Vert de_{1}\right\Vert _{\infty}\left\Vert e_{1}\right\Vert _{\infty}\left\Vert du_{A}\right\Vert _{L^{2}(M)}\left\Vert u_{A}\right\Vert _{L^{2}(M)}+\left\Vert de_{1}\right\Vert _{\infty}^{2}\left\Vert u_{A}\right\Vert _{L^{2}(M)}^{2}\right]\]
\[
+\left\Vert f\right\Vert _{\infty}^{2}\left\Vert \frac{1}{e_{1}}\right\Vert _{\infty}^{2}\int_{M}V\left|u_{A}\right|^{2}\, d\mathcal{V}_{g}\]
\[
\leq\left\Vert f\right\Vert _{\infty}^{2}\left\Vert \frac{1}{e_{1}}\right\Vert _{\infty}^{2}\left[\left\Vert du_{A}\right\Vert _{L^{2}(M)}^{2}\left\Vert e_{1}\right\Vert _{\infty}^{2}+2\left\Vert de_{1}\right\Vert _{\infty}\left\Vert e_{1}\right\Vert _{\infty}\left\Vert du_{A}\right\Vert _{L^{2}(M)}\left\Vert u_{A}\right\Vert _{L^{2}(M)}+\left\Vert de_{1}\right\Vert _{\infty}^{2}\left\Vert u_{A}\right\Vert _{L^{2}(M)}^{2}\right.\]
\[
+\left.\int_{M}\left|du_{A}\right|^{2}\, d\mathcal{V}_{g}+\int_{M}V\left|u_{A}\right|^{2}\, d\mathcal{V}_{g}\right]\]
\[
\leq\left\Vert f\right\Vert _{\infty}^{2}\left\Vert \frac{1}{e_{1}}\right\Vert _{\infty}^{2}\left[\left\Vert du_{A}\right\Vert _{L^{2}(M)}^{2}+\left\Vert V\right\Vert _{\infty}\left\Vert u_{A}\right\Vert _{L^{2}(M)}^{2}\right.\]
\[
\left.+2\left\Vert de_{1}\right\Vert _{\infty}\left\Vert e_{1}\right\Vert _{\infty}\left\Vert du_{A}\right\Vert _{L^{2}(M)}\left\Vert u_{A}\right\Vert _{L^{2}(M)}+\left\Vert de_{1}\right\Vert _{\infty}^{2}\left\Vert u_{A}\right\Vert _{L^{2}(M)}^{2}\right];\]
so, since $\left\Vert du_{A}\right\Vert _{L^{2}(M)}^{2}\leq\textrm{cap}(A)$
and $\left\Vert u_{A}\right\Vert _{L^{2}(M)}^{2}\leq\frac{\textrm{cap}(A)}{\lambda_{1}(M)}$
we get for all integer $k$, and for all function $f\in F_{k}$ such
that $\left\Vert f\right\Vert _{L^{2}(M)}=1$ :

\textit{\begin{equation}
0\leq C(f)\leq F_{k}\textrm{cap}(A)\label{eq:-1-1-3-1-1-1}\end{equation}
}where $F_{k}=F_{k}\left(f,e_{1},\lambda_{1}(M)\right)>0$. Here,
for k fixed, the constant $F_{k}$ depends also on $f$, and $f$
depends on the functions $f_{1},f_{2},\cdots,f_{k}$ (which are depends
only on $M$ and $V$) and on the scalars $\mbox{\ensuremath{\alpha}}_{1},\alpha_{2},\cdots,\alpha_{k}$;
since ${\displaystyle \sum_{i=1}^{k}\alpha_{i}^{2}}=1$, all the $\left(\alpha_{i}\right)_{1\leq i\le k}$
are bounded in $\mathbb{R}$, so finally, for all integer $k$ the
constant $F_{k}$ can be bounded by a constant (we denotes also by
$F_{k}=F_{k}(M,V$)) which depends only on $M$ and $V$.\\
$\blacklozenge$ Study of $|D(f)|$ : we have \[
|D|=\left|\int_{M}\left|df\right|^{2}v_{A}\, d\mathcal{V}_{g}+\int_{M}V\left|f\right|^{2}v_{A}\, d\mathcal{V}_{g}\right|\]
\[
\leq\left\Vert df\right\Vert _{\infty}^{2}\left\Vert \frac{1}{e_{1}}\right\Vert _{\infty}\int_{M}\left|\frac{u_{A}}{e_{1}}\right|\, d\mathcal{V}_{g}+\left\Vert \frac{V\left|f\right|^{2}}{e_{1}}\right\Vert _{\infty}\int_{M}\left|u_{A}\right|\, d\mathcal{V}_{g}\]
\[
\leq\max\left(\left\Vert df\right\Vert _{\infty}^{2}\left\Vert \frac{1}{e_{1}}\right\Vert _{\infty},\left\Vert \frac{V\left|f\right|^{2}}{e_{1}}\right\Vert _{\infty}\right)\int_{M}\left|u_{A}\right|\, d\mathcal{V}_{g}\]
\[
\leq\max\left(\left\Vert df\right\Vert _{\infty}^{2}\left\Vert \frac{1}{e_{1}}\right\Vert _{\infty},\left\Vert \frac{V\left|f\right|^{2}}{e_{1}}\right\Vert _{\infty}\right)\sqrt{\textrm{Vol}(M)}\left\Vert u_{A}\right\Vert _{L^{2}(M)}\]
\[
\leq\max\left(\left\Vert df\right\Vert _{\infty}^{2}\left\Vert \frac{1}{e_{1}}\right\Vert _{\infty},\left\Vert \frac{V\left|f\right|^{2}}{e_{1}}\right\Vert _{\infty}\right)\sqrt{\textrm{Vol}(M)}\sqrt{\frac{\textrm{cap}(A)}{\lambda_{1}(M)}}.\]
Hence, for all integer $k$, and for all function $f\in F_{k}$ such
that $\left\Vert f\right\Vert _{L^{2}(M)}=1$ :

\textit{\begin{equation}
|D(f)|\leq G_{k}\sqrt{\textrm{cap}(A)}\label{eq:-1-1-3-1-1-1-1}\end{equation}
}where (and for the same reasons as in the study of $F$, see the
constant $F_{k}$) for all integer $k$, the constant $G_{k}$ depends
only on $M$ and $V$, ie $G_{k}=G_{k}\left(M,V\right)$. \\
$\blacklozenge$ Study of $|E(f)|$ : recall that $E(f)=\int_{M}dfdv_{A}f\left(1-v_{A}\right)\, d\mathcal{V}_{g}$,
hence

\[
|E(f)|\leq\int_{M}\left|dfdv_{A}\right|\left|f\right|\, d\mathcal{V}_{g}+\int_{M}\left|dfdv_{A}\right|\left|fv_{A}\right|\, d\mathcal{V}_{g}.\]
For the first term $\int_{M}\left|dfdv_{A}\right|\left|f\right|\, d\mathcal{V}_{g}$
we have :\[
\int_{M}\left|dfdv_{A}\right|\left|f\right|\, d\mathcal{V}_{g}\leq\left\Vert f\right\Vert _{\infty}\left\Vert df\right\Vert _{\infty}\sqrt{\textrm{Vol}(M)}\left\Vert dv_{A}\right\Vert _{L^{2}(M)};\]
we have see in the study of $C(f)$ that \[
\left\Vert dv_{A}\right\Vert _{L^{2}}^{2}\]
\[
\leq\left\Vert \frac{1}{e_{1}}\right\Vert _{\infty}^{2}\left(\left\Vert du_{A}\right\Vert _{L^{2}(M)}^{2}\left\Vert e_{1}\right\Vert _{\infty}^{2}+2\left\Vert de_{1}\right\Vert _{\infty}\left\Vert e_{1}\right\Vert _{\infty}\left\Vert du_{A}\right\Vert _{L^{2}(M)}\left\Vert u_{A}\right\Vert _{L^{2}(M)}+\left\Vert de_{1}\right\Vert _{\infty}^{2}\left\Vert u_{A}\right\Vert _{L^{2}(M)}^{2}\right)\]
so with $K:=\left\Vert f\right\Vert _{\infty}\left\Vert df\right\Vert _{\infty}\sqrt{\textrm{Vol}(M)}\left\Vert \frac{1}{e_{1}}\right\Vert _{\infty}$
we get \[
\int_{M}\left|dfdv_{A}\right|\left|f\right|\, d\mathcal{V}_{g}\]
\[
\leq K\sqrt{\left\Vert du_{A}\right\Vert _{L^{2}(M)}^{2}\left\Vert e_{1}\right\Vert _{\infty}^{2}+2\left\Vert de_{1}\right\Vert _{\infty}\left\Vert e_{1}\right\Vert _{\infty}\left\Vert du_{A}\right\Vert _{L^{2}(M)}\left\Vert u_{A}\right\Vert _{L^{2}(M)}+\left\Vert de_{1}\right\Vert _{\infty}^{2}\left\Vert u_{A}\right\Vert _{L^{2}(M)}^{2}}\]

\[
\leq K\sqrt{\textrm{cap}(A)\left\Vert e_{1}\right\Vert _{\infty}^{2}+2\left\Vert de_{1}\right\Vert _{\infty}\left\Vert e_{1}\right\Vert _{\infty}\sqrt{\textrm{cap}(A)}\sqrt{\frac{\textrm{cap}(A)}{\lambda_{1}(M)}}+\left\Vert de_{1}\right\Vert _{\infty}^{2}\frac{\textrm{cap}(A)}{\lambda_{1}(M)}}\]
\[
\leq H_{k}\sqrt{\textrm{cap}(A)}\]
where (same reasons as above), for all integer $k$, the constant
$H_{k}$ depends only on $M$ and $V$, ie $H_{k}=H_{k}\left(M,V\right)$.\\
Next, for the second term : $\int_{M}\left|dfdv_{A}\right|\left|fv_{A}\right|\, d\mathcal{V}_{g}$
we have :\[
\int_{M}\left|dfdv_{A}\right|\left|fv_{A}\right|\, d\mathcal{V}_{g}\leq\left\Vert df\right\Vert _{\infty}\left\Vert f\right\Vert _{\infty}\left\Vert dv_{A}\right\Vert _{L^{2}(M)}\left\Vert v_{A}\right\Vert _{L^{2}(M)}\]
\[
\leq\left\Vert df\right\Vert _{\infty}\left\Vert f\right\Vert _{\infty}\left\Vert dv_{A}\right\Vert _{L^{2}(M)}\left\Vert \frac{1}{e_{1}}\right\Vert _{\infty}\left\Vert u_{A}\right\Vert _{L^{2}(M)}\]
\[
\leq\left\Vert df\right\Vert _{\infty}\left\Vert f\right\Vert _{\infty}\left\Vert \frac{1}{e_{1}}\right\Vert _{\infty}\sqrt{\frac{\textrm{cap}(A)}{\lambda_{1}(M)}}H_{k}\sqrt{\textrm{cap}(A)}\]
\[
\leq H_{k,M}^{\prime}\textrm{cap}(A).\]
where (same reasons as above), for all integer $k$, the constant
$H_{k}$ depends only on $M$ and $V$, ie $H_{k}^{\prime}=H_{k}\left(M,V\right)$.\\
So, for all integer $k$ :

\textit{\begin{equation}
|E(f)|\leq H_{k,M}^{\prime\prime}\left(\sqrt{\textrm{cap}(A)}+\textrm{cap}(A)\right)\label{eq:-1-1-3-1-1-1-1-1}\end{equation}
}where $H_{k}^{\prime\prime}:=H_{k}^{\prime\prime}\left(M,V\right)$.\\
Finally, with the study of $A(f),B(f),C(f),|D(f)|$ and $|E(f)|$,
for all integer $k$, for any function $\phi=f\left(1-\frac{u_{A}}{e_{1}}\right)\in E_{k}$,
with $f\in F_{k}$ such that $\left\Vert f\right\Vert _{L^{2}(M)}=1$
we get :

\textit{\begin{equation}
\int_{M}\left|d\phi\right|^{2}d\mathcal{V}_{g}+\int_{M}V\left|\phi\right|^{2}d\mathcal{V}_{g}\leq\lambda_{k}(M)+I_{k}\left(\sqrt{\textrm{cap}(A)}+\textrm{cap}(A)\right)\label{eq:-1-1-3-1-1-1-1-1-1}\end{equation}
}where, for all integer $k$, the constant $I_{k}$ depends only on
$M$ and $V$, ie : $I_{k}=I_{k}\left(M,V\right)$.\\
\\
$\bullet$\textbf{ Step 3} : Now we claim that : for all $A\subset M$
such that $\textrm{cap}(A)\leq\varepsilon_{k}$ and for any function
$\phi\in E_{k}$ we have :\textit{\begin{equation}
\left\Vert \phi\right\Vert _{L^{2}(M)}^{2}\geq1-J_{k,M}^{\prime}\sqrt{\textrm{cap}(A)}\label{eq:-1-1-3-1-1-1-1-1-1-1}\end{equation}
}where, for all integer $k$, the constant $J_{k,M}^{\prime}$ depend
only on $M$ and $V$, ie : $J_{k,M}^{\prime}=J_{k,M}^{\prime}\left(M,V\right)$.
\\
Indeed : let $\phi\in E_{k}$, we have seen below in step 1 that
:\[
\textrm{cap}(A)\leq\varepsilon_{k}\Rightarrow\textrm{dim}(E_{k})=k\;\textrm{and}\;\forall j\in\{1,...,k\},\,\left|\left\Vert \phi_{j}\right\Vert _{L^{2}(M)}^{2}-1\right|\leq D_{k}\sqrt{\textrm{cap}(A)}\]
therefore, since $\phi\in E_{k}$, we can write $\phi=(1-v_{A})f$
with $f={\displaystyle \sum_{i=1}^{k}\alpha_{i}e_{i}}$ where $\left(\alpha_{i}\right)_{1\leq i\leq k}\in\mathbb{R}^{k}$.
As in the step two we can assume that $\left\Vert f\right\Vert _{L^{2}(M)}=1$,
hence we have ${\displaystyle \sum_{i=1}^{k}\alpha_{i}^{2}}=1$. Next,
compute $\left\Vert \phi\right\Vert _{L^{2}(M)}^{2}$ :\[
\left\Vert \phi\right\Vert _{L^{2}(M)}^{2}=\left\Vert \sum_{i=1}^{k}\left(1-v_{A}\right)\alpha_{i}e_{i}\right\Vert _{L^{2}(M)}^{2}=\left\Vert \sum_{i=1}^{k}\alpha_{i}\phi_{i}\right\Vert _{L^{2}(M)}^{2}\]
\[
=\sum_{i=1}^{k}\alpha_{i}^{2}\left\Vert \phi_{i}\right\Vert _{L^{2}(M)}^{2}+\sum_{i,j\; i\neq j}\alpha_{i}\alpha_{j}\left\langle {\displaystyle \phi_{i}},{\displaystyle \phi_{j}}\right\rangle _{L^{2}(M)}.\]
And since\[
\sum_{i=1}^{k}\alpha_{i}^{2}\left\Vert \phi_{i}\right\Vert _{L^{2}(M)}^{2}=\sum_{i=1}^{k}\alpha_{i}^{2}\left[1-2\int_{M}e_{i}^{2}v_{A}\, d\mathcal{V}_{g}+\int_{M}e_{i}^{2}v_{A}^{2}\, d\mathcal{V}_{g}\right]\]
\[
=1-\sum_{i=1}^{k}\alpha_{i}^{2}\left[2\int_{M}e_{i}^{2}v_{A}\, d\mathcal{V}_{g}-\int_{M}e_{i}^{2}v_{A}^{2}\, d\mathcal{V}_{g}\right]\]
\[
=1-\sum_{i=1}^{k}\alpha_{i}^{2}\int_{M}e_{i}^{2}\left(2v_{A}-v_{A}^{2}\right)\, d\mathcal{V}_{g};\]
hence \[
\left\Vert \phi\right\Vert _{L^{2}(M)}^{2}=1-\sum_{i=1}^{k}\alpha_{i}^{2}\int_{M}e_{i}^{2}\left(2v_{A}-v_{A}^{2}\right)\, d\mathcal{V}_{g}+\sum_{i,j\; i\neq j}\alpha_{i}\alpha_{j}\left\langle {\displaystyle \phi_{i}},{\displaystyle \phi_{j}}\right\rangle _{L^{2}(M)}\]
we have seen in step 1 that, for $\textrm{cap}(A)$ small enough :\[
\left|\left\langle \phi_{i},\phi_{j}\right\rangle _{L^{2}(M)}-\delta_{i,j}\right|\leq B_{k}\left(\sqrt{\textrm{cap}(A)}+\textrm{cap}(A)\right)\]
hence, since all the $\left(\alpha_{i}\right)_{1\leq i\le k}$ are
bounded in $\mathbb{R}$, and for $\textrm{cap}(A)$ small enough,
we can find a constant $B_{k,M}^{\prime}$ which depends only on $M$
and $V$, ie $B_{k}^{\prime}=B_{k}^{\prime}\left(M,V\right)$ such
that, for $\textrm{cap}(A)$ small enough :\[
\left|\sum_{i,j\; i\neq j}\alpha_{i}\alpha_{j}\left\langle {\displaystyle \phi_{i}},{\displaystyle \phi_{j}}\right\rangle _{L^{2}(M)}\right|\leq B_{k}^{\prime}\sqrt{\textrm{cap}(A)}\]
and finally, in the same spirit as in the estimations in section 2,
there exists a constant $B_{k,M}^{\prime\prime}$ which depends only
on $M$ and $V$, ie $B_{k}^{\prime\prime}=B_{k}^{\prime\prime}\left(M,V\right)$
such that, for $\textrm{cap}(A)$ small enough : \[
\left|\sum_{i=1}^{k}\alpha_{i}^{2}\int_{M}e_{i}^{2}\left(2v_{A}-v_{A}^{2}\right)\, d\mathcal{V}_{g}\right|\leq B_{k}^{\prime\prime}\sqrt{\textrm{cap}(A)}\]
so finally we obtain :\[
\left\Vert \phi\right\Vert _{L^{2}(M)}^{2}\geq1-B_{k}^{\prime\prime\prime}\sqrt{\textrm{cap}(A)}\]
where the constant $B_{k}^{\prime\prime\prime}$ depend only on $M$
and $V$, ie : $B_{k}^{\prime\prime\prime}:=B_{k}^{\prime\prime\prime}\left(M,V\right)$.\\
\\
$\bullet$\textbf{ Final step} : As a consequence from step 2 and
3, for all function $\phi\in E_{k}$ we get :\[
\frac{\int_{M}\left|d\phi\right|^{2}\, d\mathcal{V}_{g}+\int_{M}V\left|\phi\right|^{2}\, d\mathcal{V}_{g}}{\int_{M}\phi^{2}\, d\mathcal{V}_{g}}\leq\frac{\lambda_{k}(M)+I_{k}\left(\textrm{cap}(A)+\sqrt{\textrm{cap}(A)}\right)}{1-B_{k}^{\prime\prime\prime}\sqrt{\textrm{cap}(A)}}\]
hence for $\textrm{cap}(A)$ small enough (ie : $\textrm{cap}(A)\leq\varepsilon_{k}$)
we have\[
\frac{\int_{M}\left|d\phi\right|^{2}\, d\mathcal{V}_{g}+\int_{M}V\left|\phi\right|^{2}\, d\mathcal{V}_{g}}{\int_{M}\phi^{2}\, d\mathcal{V}_{g}}\leq\lambda_{k}(M)+L_{k}\sqrt{\textrm{cap}(A)}\]
where $L_{k}:=L_{k}\left(M,V\right)$. Next, since for all $k\geq1$\[
\lambda_{k}(M-A)=\underset{\underset{\dim(E)=k}{E\subset H_{0}^{1}(M-A)}}{\mathcal{\textrm{min}}}\underset{\underset{\varphi\neq0}{\varphi\in E}}{\mathcal{\textrm{max}}}\frac{\int_{M}\left|d\varphi\right|^{2}\, d\mathcal{V}_{g}+\int_{M}V\left|\varphi\right|^{2}\, d\mathcal{V}_{g}}{\int_{M}\varphi^{2}\, d\mathcal{V}_{g}}\]
and since $\phi\in H_{0}^{1}(M-A)$, we get for all $k\geq1$\[
\lambda_{k}(M-A)\leq\frac{\int_{M}\left|d\phi\right|^{2}\, d\mathcal{V}_{g}+\int_{M}V\left|\phi\right|^{2}\, d\mathcal{V}_{g}}{\int_{M}\phi^{2}\, d\mathcal{V}_{g}}\leq\lambda_{k}(M)+C_{k}\sqrt{\textrm{cap}(A)}.\]
And the statement of the theorem is established.\end{proof}

\vspace{1cm}

\hspace{-0.5cm}\textbf{\large Olivier Lablée}{\large \par}

\hspace{-0.5cm}

\hspace{-0.5cm}Université Grenoble 1-CNRS

\hspace{-0.5cm}Institut Fourier

\hspace{-0.5cm}UFR de Mathématiques

\hspace{-0.5cm}UMR 5582

\hspace{-0.5cm}BP 74 38402 Saint Martin d'Hères 

\hspace{-0.5cm}France

\hspace{-0.5cm}mail: \textcolor{blue}{olivier.lablee@ac-grenoble.fr}

\hspace{-0.5cm}http://www-fourier.ujf-grenoble.fr/\textasciitilde{}lablee/
\end{document}